\documentclass[12pt, twoside]{article}
\usepackage{amsmath,amsthm,amssymb}
\usepackage{times}
\usepackage{enumerate}
\usepackage{url}

\usepackage{color} % Use color

\usepackage{verbatim} % Multi-line comments

\usepackage{mathrsfs} % For \mathscr, \mathscr

\def\titlerunning#1{\gdef\titrun{#1}}
\makeatletter
\def\author#1{\gdef\autrun{\def\and{\unskip, }#1}\gdef\@author{#1}}
\def\address#1{{\def\and{\\\hspace*{18pt}}\renewcommand{\thefootnote}{}%
\footnote {#1}}%
\markboth{\autrun}{\titrun}}
\makeatother
\def\email#1{e-mail: #1}
\def\subjclass#1{{\renewcommand{\thefootnote}{}%
\footnote{\emph{Mathematics Subject Classification (2010):} #1}}}
\def\keywords#1{\par\medskip
\noindent\textbf{Keywords.} #1}

\newtheorem{thm}{Theorem}[section]

\newtheorem{lem}[thm]{Lemma}

\newtheorem{mainthm}[thm]{Main Theorem}

\theoremstyle{definition}
\newtheorem{defin}[thm]{Definition}

\newtheorem*{xrem}{Remark}

\numberwithin{equation}{section}

\begin{document}

%%%%% To ease editing, add:

\baselineskip=17pt

%%%%%%%%%%%%%%%%

%% In the running head, give an abbreviation of the title. 
\titlerunning{Witten Deformation and Its Application toward Morse Inequalities}

\title{Witten Deformation and Its Application toward Morse Inequalities}

\author{Fu-Hsuan Ho}

\date{}

\maketitle

\begin{abstract}
In this undergraduate thesis, we present an analytical proof of the Morse inequalities for closed smooth $n$-manifolds following Witten's approach. Using techniques from PDE theory, the proof is reduced to study the eigenspaces and eigenvalues of harmonic oscillators on $\mathbb{R}^n$.
\end{abstract}

\begin{comment}
\address{F1. Surname1: address1; \email{b02602018@ntu.edu.tw}
\and
F2. Surname2: address2; \email{xx2@yy2}}

\subjclass{Primary MMMM; Secondary YYYY}

\begin{abstract}
A template for submissions to JEMS. For the subject classification, use
the 2010 Mathematics Subject Classification available at www.ams.org/msc.

%% Keywords are optional
\keywords{Morse inequalitites, Witten deformation}
\end{abstract}
\end{comment}

\section{Introduction}
Morse theory has been a branch of differential topology that provides a direct way to understand the topology of smooth manifolds by studying the smooth functions on them. In the introduction of the classical book \cite{Milnor}, Milnor discussed a toy model of this idea: assume $M=T^2$ is a torus tangent to a plane $V$ and $M^a$ be the set of points $p\in M$ so that $f(p)\leq a$. and let $f:M\rightarrow \mathbb{R}$ be the height above $V$. Milnor pointed out that the homotopy type changes exactly at the critical points of $f$. These critical points are all ``non-degenerate'' (see section 2 for definition), and near each critical point $p$ one can choose proper coordinate $(x,y)$ so that $f=f(p)\pm x^2\pm y^2$. Note that the number of minus signs in the expression for $f$ at each point is the dimension of the cell we must attach to go from $M^a$ to $M^b$, where $a < f(p) < b$. This gives us a motivation to study critical points of smooth functions on manifolds.

Another formula also shows the deep connection of the topology and critical points of smooth functions on a manifold, that is, the morse inequalities.  Let $M$ be a $n$-dimensional closed and smooth manifold. Given $f\in C^{\infty}(M)$, we say $f$ is a Morse function if its critical points are non-degenerate (these terms will be defined in section 2). Morse inequalities give a series of upper bound related to the Betti number of $M$. The full statement is given in the following:

\begin{mainthm}[Morse Inequalities]
Let $M$ be a closed smooth $n$-dimensional manifold. Suppose $f\in C^{\infty}(M)$ is a Morse function. Then we have
\begin{itemize}
    \item \textbf{Weak Morse Inequalities:} for each $0\leq q\leq n$, 
    \begin{align}
      b_q \leq m_q
    \end{align}
    
    \item \textbf{Strong Morse Inequalities:} for each $0\leq q\leq n$
       \begin{align}
         \sum^q_{j = 0} (- 1)^{q - j} b_j \leq \sum_{j = 0}^q (- 1)^{q - j} m_j,
       \end{align}
    and the equality holds when $q = n$, i.e.,
       \begin{align}\label{E:Main}
         \sum^n_{j = 0} (- 1)^{n - j} b_j = \sum_{j = 0}^n (- 1)^{n - j} m_j,
       \end{align}
\end{itemize}
where we denote $m_j$ as the cardinality of the set of critical points of order $j$ and $b_j$ be the $j$-th Betti number of $M$ (see (\ref{E:Betti})), for any $j\in\{0,\ldots,q\}$.
\end{mainthm}

Note that (\ref{E:Main}) gives a way to derive $\chi(M)$ by calculating the alternating sum of numbers of critical points up to index $n$, since the definition of Euler characteristic number $\chi(M)$ of $M$ is exactly the alternating sum of Betti numbers.

A topological proof may be founded in \cite{Milnor}. However in 1982, Witten \cite{Witten} discovered the first analytic approach to the Morse inequalities through considering the de Rham complex the deformed exterior derivative $d_t=e^{-tf}d e^{tf}$, i.e., the so-called Witten deformation. Inspired by physics, he adopted the semi-classical analysis of the eigenvalues of some properly chosen Schr\"{o}dinger operators on $M$ to attack this problem. His idea was immediately absorbed by many mathematicians and was turned into mathematically rigor papers, for example, \cite{Helffer}. Another reference is \cite{Cycon}. 

In this paper, we try to provide a rather direct and detailed proof that can be accessed by undergraduate students who have the basic knowledge of geometry. In section 2, we recall the basic notations and preliminaries that are needed in this thesis. Then in section 3, we introduce to the Witten Laplacian and reduce the proof of Morse inequality to studying the dimension of the eigenspace of Witten Laplacian, and we do this work in section 4. The strategy is constructing enough linear independent $q$-forms on $M$ (see subsection 4.2).  Section 5 is devoted to calculating the eigenspace of Witten Laplacian. While the calculation is essential to our proof, the method we adopt to find eigenvalues is rather standard and somehow tedious, and thus we put it in the last section.

The last thing to be notice is that although we only consider the case of closed manifolds, this method can be extended to more general cases, for example, manifolds with boundary. Also, the idea of constructing linearly independent elements of operator's eigen-space can be extended to understand other operators we are interested in.
\section{Notations and Preliminaries}

From now on, we fix $M$ to be a smooth and closed n-dimensional manifold. Let $C^{\infty}(M)$ be the set of real-valued smooth functions of $M$. A \textbf{critical point} of a given smooth function $f$ is a point $p$ so that for any local coordinate chart $(U,\varphi,x)$ near $p$,
\begin{align}
  \frac{\partial f}{\partial x_1}(p)=\cdots=\frac{\partial f}{\partial x_n}(p)=0.
\end{align}
This definition is independent of our choice of coordinate chart near $p$. For any other local coordinate chart $(V,\phi,y)$ near $p$, using the change of coordinate formula we get,
\begin{align}
  \frac{\partial f}{\partial y_i}(p)=\frac{\partial f}{\partial x_j}(p)\frac{\partial (\phi\circ\varphi^{-1})}{\partial y_i}(p)=0,\quad\forall i=1,\ldots,n.
\end{align}
A critical point $p\in M$ of $f$ is called \textbf{non-degenerate} if and only if the matrix 
\begin{align}\label{E:non-degenerate}
  \left( \frac{\partial^2 f}{\partial x_i \partial x_j}(p)\right)
\end{align}
is invertible. We quickly check that this definition is independent of the choice of the coordinate system: let $(U,\varphi,x)$ be a local chart satisfying (\ref{E:non-degenerate}). Assume $(V,\phi,y)$ is any other local chart with non-empty intersection with $U$. Then
\begin{align}\label{E:non-degenerate-CoC}
  \frac{\partial f}{\partial x_i\partial x_j}(p) &= \sum_{k,l}\frac{\partial (\phi\circ\varphi)^{-1}_k}{\partial x_i}(p)\frac{\partial f}{\partial y_k\partial y_l}(p)\frac{\partial (\phi\circ\varphi)^{-1}_l}{\partial x_j}(p).
\end{align}
Because $\phi\circ\varphi^{-1}$ is diffeomorphism, the matrix
\[
  \left(\frac{\partial (\phi\circ\varphi)^{-1}_k}{\partial x_i}(p)\right)
\]
is invertible. Therefore, from \ref{E:non-degenerate-CoC} we must derive that 
\[
  \left(\frac{\partial^2 f}{\partial y_k\partial y_l}(p)\right)
\]
is also invertible. With the terminologies prepared above, we now give the definition of Morse function:
\begin{defin}
  We say smooth function $f\in C^\infty (M)$ is a Morse function if its critical points are all non-degenerate.
\end{defin}
If $p\in M$ is a critical point, we define a symmetric bilinear form 
\[d^2 f:T_pM\times T_pM\rightarrow\mathbb{R}\]
called the \textbf{Hessian of $f$ at $p$} on $T_pM$, the tangent space of $M$ at $p$. If $X,Y\in T_p(M)$, we may extend them to smooth vector fields $\tilde{X},\tilde{Y}$ with $\tilde{X}(p)=X$ and $\tilde{Y}(p)=Y$. Define
\[
  d^2 f(X,Y)=\tilde{X}(\tilde{Y}(f))(p).
\]
In the following we will show that $d^2f$ is a well-defined symmetric bilinear functional: note $\tilde{X}(p)$ equals to $X$. Indeeds, by the formula
\[
  \tilde{X}(p)(\tilde{Y}(f))-\tilde{Y}(p)(\tilde{X}(f))=df[\tilde{X},\tilde{Y}](p)=0,
\]
$d^2 f$ is symmetric. Also by the formula above, we have 
\[
  X(\tilde{Y}(f))=\tilde{X}(p)(\tilde{Y}(f))=\tilde{Y}(p)(\tilde{X}(f))=Y(\tilde{X}(f))
\]
and thus $d^2 f$ is independent of how we extend $X$ and $Y$, i.e., it is well-defined.

Locally on a chart $(U,\varphi,x)$ near $p$, we may write $X=\sum_{i=1}^na_i \frac{\partial}{\partial x_i}|_p$ and $Y= \sum_{j=1}^nb_j \frac{\partial}{\partial x_j}|_p$. Locally, we may extend $Y$ by $\tilde{Y}=\sum_{j=1}^nb_j \frac{\partial}{\partial x_j}$. Then
\begin{align}\label{E:d^2f}
  d^2 f(X,Y)=X(\tilde{Y} f)(p)=X\left(\sum_{j=1}^nb_j \frac{\partial}{\partial x_j}\right)=\sum_{i=1,j=1}^n a_ib_j\frac{\partial^2 f}{\partial x_i\partial x_j}(p);
\end{align}
so the matrix
\begin{align}\label{d^2f-matrix-representation}
  \left(\frac{\partial^2 f}{\partial x_k\partial x_l}(p)\right)
\end{align}
represents the bilinear form $d^2 f$ under the basis $\{\frac{\partial}{\partial x_i}|_p\}_{i=1}^n$.

  Now, we mention the index of $d^2 f$. For any bilinear form $B$ defined on a vector space $\mathbf{V}$, the \textbf{index} of $B$ is defined to be the maximal dimension of any subspace $\mathbf{W}$ on which $B$ is negative definite. When we say the phrase ``\textbf{the index of $f$ at $p$}'', it always refers to the index of $d^2 f$ on $T_pM$. We can check that the index of $f$ at $p$ is the number of non-negative eigenvalues of the matrix (\ref{d^2f-matrix-representation}). The lemma of Morse, which will be stated in the following (for the proof, see \cite[Lemma 2.2]{Milnor}), shows that the index of $f$ at $p$ characterizes the local behavior of $f$ near $p$:
  
\begin{lem}[Morse]\label{lem:Morse}
  Let $f$ be a smooth function. Assume $p\in M$ is a critical point of $f$ on $M$. Then we can find a local chart $(U,\varphi,x)$ near $p$ with $\varphi(p)=0$ so that 
  \[ f\circ\varphi^{-1} (x) = - \frac{1}{2} x^2_1 - \cdots - \frac{1}{2} x_l^2 + \frac{1}{2}
     x_{l + 1}^2 + \cdots + \frac{1}{2} x_n^2, \]
where $l$ is the index of $f$ at $p$.
\end{lem}  

From now on, we always assume $f$ to be a fixed Morse function on $M$. We introduce the notation
\[\text{Crit}(f)\]
as the set of critical points of $f$. By Morse lemma, the set $\text{Crit}(f)$ must be discrete, and it turns out to be finite by compactness of $M$. We also introduce the notation
\[\text{Crit}(f;j)\]
to represent the set of critical points on which the index of $f$ equals to $j$, and we define
\[m_j=\#\text{Crit}(f;j)\]
to be the cardinality of $\text{Crit}(f;j)$.
\subsection{de Rham Cohomology}
In this subsection, we discuss the basic properties of de Rham cohomology. The readers may consult Bott and Tu's book \cite{BottTu} for more information. 

Denote $\Omega^0(M)$ be the space of smooth functions. For any $k=1,\ldots,n$, let $\Lambda^k T^*M$ be the bundle of $k$ forms:
\[
  \Lambda^k T^* M=\coprod_{p\in M}\Lambda^kT^*_pM.
\]
Fix $p\in M$ and let $(U,\varphi,x)$ be any local chart defined near $p$. Assume $\{dx_1,\ldots,dx_n\}$ is the basis for $T^*_pM$. Then the fibre of $\Lambda^k T^*M$ at $p$ is the vector space $\Lambda^kT^*_pM$ spanned by
\[\{dx_{i_1}\wedge\cdots\wedge d_{i_k}:1\leq i_1<\cdots<i_k\leq n\}.\]

A smooth section of $\Lambda^k T^* M$ is called a \textbf{smooth $k$-forms}. We denote 
\[\Omega^k(M)=\Gamma\left(\Lambda^k T^* M\right)\]
as the vector space of smooth $k$-forms. On any local chart $(U,\varphi,x)$, a smooth $k$-form $\omega$ can be written as
\begin{align}\label{E:Dform-local-form}
  \omega=\sum_{i_1<\cdots<i_k} \omega_{i_1\cdots i_k} dx_{i_1}\wedge\cdots\wedge dx_{i_k},
\end{align}
where $\omega_{i_1\cdots i_k}$ is a smooth function defined on $U$. Once again, the local representation (\ref{E:Dform-local-form}) is independent of choice of local chart. Suppose on $(\tilde{U},\tilde{\varphi},\tilde{x})$ with $\tilde{U}\cap U\neq \emptyset$. Then there exists

The exterior derivative $d$ is defined to be the unique real linear mapping
\[d:\Omega^q(M)\rightarrow\Omega^{q+1}(M)\]
sending smooth $q$-forms to $q+1$-forms, for all $q=0,\ldots,n$, satisfying the following properties
\begin{enumerate}[(i)]
\item $df$ is the differential of smooth function for any $f\in \Omega^0(M)$.
\item $d^2 f=0$ for any $f\in \Omega^0(M)$.
\item $d(\alpha\wedge\beta)=d\alpha\wedge\beta+(-1)^p\alpha\wedge d\beta$ where $\alpha$ is any $p$-form and $\beta$ is any $r$-form.
\end{enumerate}
If there is a need to distinguish those $d$s acting on different form, we will add an upper index $(q)$ and write $d^{(q)}$. 

The complex
\[ 0 \rightarrow \Omega^0 (M) \xrightarrow{d} \Omega^1 (M) \xrightarrow{d}
   \cdots \xrightarrow{d} \Omega^n (M) \rightarrow 0. \]
together with the exterior derivative $d$ is called the \textbf{de Rham complex on $M$}. 

Note $d^2$ always vanish. We verify this in the following: let $\omega\in\Omega^{q}$. Then on any local chart $(U,\varphi,x)$, we can write 
\[\omega=\sum_{i_1<\cdots<i_q}\omega_{i_1\cdots i_q} dx^{i_1}\wedge\cdots\wedge dx^{i_q}.\]
Applying $d$ twice on $\omega$, we get
\begin{align*}
d^2\omega=\sum_{i,j=1}^n\sum_{i_1<\cdots<i_q}\frac{\partial^2\omega_{i_1\cdots i_q}}{\partial x_i\partial x_j} dx^i\wedge dx^j\wedge dx^{i_1}\wedge\cdots\wedge dx^{i_q} \\
=\sum_{i<j}\sum_{i_1<\cdots<i_q}\left(\frac{\partial^2\omega_{i_1\cdots i_q}}{\partial x_i\partial x_j}-\frac{\partial^2\omega_{i_1\cdots i_q}}{\partial x_j\partial x_i}\right) dx^i\wedge dx^j\wedge dx^{i_1}\wedge\cdots\wedge dx^{i_q}=0
\end{align*}
since $\omega_{i_1\cdots i_q}$ is a smooth function and hence we can change the order of derivations.

Thus, it is able to define the $j$-th de Rham cohomology
\[ H^j (M) = \frac{Ker (d : \Omega^j (M) \rightarrow \Omega^{j + 1}
   (M))}{Im (d : \Omega^{j - 1} (M) \rightarrow \Omega^j (M))} \]
for all $j = 0, \cdots, n$. Since our manifold $M$ is has finite ``good cover'' (in the sense of Bott and Tu), $H^j (M)$ is a finite
dimensional (real) vector space (the readers may consult their book \cite[Proposition 5.3.1]{BottTu}). As a result, it is able to define
\begin{align}\label{E:Betti}
  b_j = \dim H^j (M)
\end{align}
called the \textbf{$j$-th Betti number} for all $j = 0, \cdots, n$.

Fix $g$ to be a Riemannian metric on the tangent bundle $T M$. It also induces a bundle metric pointwisely on $\alpha, \beta \in\Lambda^r T_p^* M$ (independent of basis we chosen):
\[
  g(\alpha,\beta)=g^{i_1j_1}\cdots g^{i_kj_k}\alpha_{i_1\ldots i_k}\beta_{j_1\ldots j_k}.
\]
which we also denoted by $g$.
Let $dv_g$ be the density on $M$ induced by $g$ so that for any chart $(U,\phi,x)$, $w \in C_0^{\infty} (U)$,
\[ \int_M w\, d v_g = \int_{\phi(U)} w\circ\phi^{-1} (x) \sqrt{|\det (\phi_{*}g)|} d
   x_1\wedge\cdots\wedge dx_n. \]
We remark that if $M$ is orientable, then it is no need to take absolute value above.
Through the density $dv_g$, we are then able to defined a global $L^2$ inner product by
\[ (\alpha, \beta) := \int_M g(\alpha, \beta)\,d v_g. \]
Under this inner product $(\cdot,\cdot)$, we denote $d^{*(q)}$ to be the Hilbert space adjoint of $d^{(q)}$ for all $q=0,\ldots,n$. That is,
\[
  (d u,v)=(u, d^{*} v)
\]
for all $u\in \Omega^q(M)$ and $v\in \Omega^{q+1}(M)$ (once more, we will drop the upper index $(q)$ and write $d^*$ if there is no need to clarify which space $d^{*(q)}$ is acting on). We call the operator
\[
  \triangle^{(q)}:=dd^*+d^*d:\Omega^{q}(M)\rightarrow\Omega^{q}(M)
\]
the Hodge Laplacian for all $q=0,\ldots,n$. Note that this operator depends on the Riemannian metric $g$ since $d^*$ depends on the inner product $(\cdot,\cdot)$ and this inner product is determined by $g$.

We end this section with two lemmas about the local form of $d$ and $d^*$:

\begin{lem}\label{lem:d-local}
Fix $q\in\{0,\ldots,n-1\}$. For any $\omega\in\Omega^q(M)$, we have
\begin{align}\label{E:d-local}
d\omega=\sum_{i=1}^ndx_i\wedge\nabla_{\frac{\partial}{\partial x_i}} \omega.
\end{align}
\end{lem}
\begin{xrem} $ $
  \begin{enumerate}[(i)]
    \item The operator $\nabla:\Gamma(\Lambda^qT^*M)\rightarrow \Gamma(\Lambda^qT^*M\otimes T^*M)$ is natural extension of the Levi-civita connection $\nabla^{TM}:\Gamma(TM)\rightarrow\Gamma(TM\otimes TM)$.
    \item The notation $dx_k\wedge$ represents the exterior multiplications by $dx_k$, respectively. On any local chart, $dx_k\wedge$ sends an given $r$-form $\alpha$, $r\in\{0,\ldots,n\}$ to $dx_k\wedge\alpha$. The interior multiplication acts on differential forms following the rules:
  \begin{enumerate}
    \item if $\alpha$ is a $1$-form, then 
    \[dx_k\lrcorner\alpha=\left\langle\alpha,\frac{\partial}{\partial x_k}\right\rangle\]
    where $\langle\cdot,\cdot\rangle$ is the dual pairing;
    \item if $\beta$ is a smooth $p$-form and $\gamma$ is a smooth $q$-form, then
    \[dx_k\lrcorner(\beta\wedge\gamma)=(dx_k\lrcorner\beta)\wedge\gamma+(-1)^p\beta\wedge (dx_k\lrcorner\gamma).\]
  \end{enumerate}
  \end{enumerate}
\end{xrem}
\begin{proof}
  Note that both sides of (\ref{E:d-local}) are section of $\Lambda^{q+1}T^*M$. So we can check that they agree with each other on each $p\in M$. Fix $p\in M$ and we adopt the normal geodesic coordinate around $p$. Say near $p$, we have
  \[
    \omega=\sum_{i_1<\cdots <i_q}\omega_{i_1\cdots i_q} dx_{i_1}\wedge\cdots\wedge dx_{i_q}.
  \]
Then 
\begin{align}
\begin{split}
d\omega=\sum_{i_1<\cdots< i_q}\sum_{i=1}^n\frac{\partial\omega_{i_1\cdots i_q}}{\partial x_i} dx_i\wedge dx_{i_1}\wedge\cdots\wedge dx_{i_q} \\
=\sum_{i_1<\cdots< i_q}\sum_{i=1}^n \nabla_{\frac{\partial}{\partial x_i}}\omega_{i_1\cdots i_q} dx_i\wedge dx_{i_1}\wedge\cdots\wedge dx_{i_q} \\
=\sum_{i_1<\cdots< i_q}\sum_{i=1}^n dx_i\wedge\nabla_{\frac{\partial}{\partial x_i}}(\omega_{i_1\cdots i_q} dx_{i_1}\wedge\cdots\wedge dx_{i_q}) \\
=\sum_{i=1}^n dx_i\wedge\nabla_{\frac{\partial}{\partial x_i}}\omega.
\end{split}
\end{align}
\end{proof}

\begin{lem}\label{lem:d^*-local}
Fix $q\in\{1,\ldots,n\}$. For any $\omega\in\Omega^q(M)$, we have
\begin{align}\label{E:d^*-local}
d^*\omega=-\sum_{i=1}^ndx_i\lrcorner\nabla_{\frac{\partial}{\partial x_i}} \omega.
\end{align}
\end{lem}
\begin{xrem}
The notation $dx_k\lrcorner$ represents the interior multiplications by $dx_k$. On any local chart, the interior multiplication acts on differential forms following the rules:
  \begin{enumerate}
    \item if $\alpha$ is a $1$-form, then 
    \[dx_k\lrcorner\alpha=\left\langle\alpha,\frac{\partial}{\partial x_k}\right\rangle\]
    where $\langle\cdot,\cdot\rangle$ is the dual pairing;
    \item if $\beta$ is a smooth $p$-form and $\gamma$ is a smooth $q$-form, then
    \[dx_k\lrcorner(\beta\wedge\gamma)=(dx_k\lrcorner\beta)\wedge\gamma+(-1)^p\beta\wedge (dx_k\lrcorner\gamma).\]
  \end{enumerate}
\end{xrem}
\begin{proof}
With the same reason of the proof lemma \ref{lem:d-local}, we can restrict our calculation on a fixed point $p\in M$. Again, use the normal geodesic coordinate around $p$. Also, write
\[
\omega=\sum_{i_1<\cdots <i_q}\omega_{i_1\cdots i_q} dx_{i_1}\wedge\cdots\wedge dx_{i_q}.
\]
Applying lemma \ref{lem:d-local}, we have for any $u\in \Omega^{q-1}(M)$,
\begin{align}\label{E:2.10}
\begin{split}
g(\omega,du)=g(\omega,\sum_{i=1}^n dx_i\wedge \nabla_{\frac{\partial}{\partial x_i}} u)=\sum_{i=1}^n g(dx_i\lrcorner\omega, \nabla_{\frac{\partial}{\partial x_i}} u) \\
=\sum_{i=1}^n\frac{\partial}{\partial x_i}g(dx_i\lrcorner \omega,u)-\sum_{i=1}^ng(\nabla_\frac{\partial}{\partial x_i}dx_i\lrcorner\omega,u) \\
=\sum_{i=1}^n\nabla_{\frac{\partial}{\partial x_i}}g(dx_i\lrcorner \omega,u)-\sum_{i=1}^ng(\nabla_\frac{\partial}{\partial x_i}dx_i\lrcorner\omega,u)
\end{split}
\end{align}
Integrating the both sides of formula (\ref{E:2.10}) and noticing that $(d^*\omega,u)=(\omega,du)$, we then derive that 
\begin{align}
(d^*\omega,u)=(\sum_{i=1}^n-\nabla_{\frac{\partial}{\partial x_i}}dx_i\lrcorner\omega,u)
\end{align}
for all $u\in \Omega^{q-1}(M)$, i.e., $d^*\omega=\sum_{i=1}^n-\nabla_{\frac{\partial}{\partial x_i}}dx_i\lrcorner\omega$.
\end{proof}

\section{Witten Deformation of Laplacian}

This section is divided into the following parts: first, we will introduce the deformed de Rham cohomologies. We will show that they are all isomorphic to the usual de Rham cohomologies. Next, we will define the Witten Laplacian. We will prove a similar inequalities like the one of Morse for the deformed de Rham cohomologies. The only difference is the upper bound of cohomology dimension are replaced by things related to the dimension of eigenspaces of Witten Laplacian.

\subsection{Deformed de Rham Cohomology}

Consider the following deformed exterior derivatives
\[d_t^{(q)}:\Omega^{q}(M)\rightarrow\Omega^{q+1}(M)\]
defined by
\[ d_t^{(q)} = e^{- t f} d^{(q)} e^{t f}, \quad \forall t \geqslant 0 \]
for all $q=0,\ldots,n$. If the content is clear, we will drop the upper index of $d_t^{(q)}$ and write $d_t$. We then get the deformed complex
\[ 0 \rightarrow \Omega^0 (M) \xrightarrow{d_t} \Omega^1 (M) \xrightarrow{d_t}
   \cdots \xrightarrow{d_t} \Omega^n (M) \rightarrow 0, \quad \forall t
   \geqslant 0. \]
Since $d^2_t = e^{- t f} d^2 e^{t f} = 0$, we can also define $H^{\bullet}_t
(M)$ the cohomology of this complex. Note that the defomation does not change
the cohomology data. To be precise,

\begin{thm}
  The map of multiplication
  \[ \begin{array}{cccc}
       e^{- t f} : & H^j (M) & \rightarrow & H^j_t (M)\\
       & u + \text{Im} d & \mapsto & e^{- t f} u + \text{Im} d_t
     \end{array} \]
  is a vector space isomorphism for all $j = 0, \cdots, n$.
\end{thm}

\begin{proof}
  Fix $j = 0, \cdots, n$. It suffices to show $e^{- t f}$ is a well-defined
  bijection because linearity is quite obvious.
  \begin{itemize}
    \item \textbf{$e^{- t f}$ is a well-defined map:} let
    \[ u \in \text{Im} (d : \Omega^{j - 1} (M) \rightarrow \Omega^j (M)), \]
    i.e., $u = d v$ for some $v \in \Omega^{j - 1} (M)$. Since
    \[ e^{- t f} u = e^{- t f} d v = d_t (e^{- t f} v) \]
    and $e^{- t f} v \in \Omega^{j - 1} (M)$,
    \[ e^{- t f} u \in \text{Im} (d_t : \Omega^{j - 1} (M) \rightarrow
       \Omega^j (M)) . \]
    \item \textbf{$e^{- t f}$ is injective:} assume $e^{- t f} u \in \text{Im} d_t$,
    i.e., $e^{- t f} u = d_t v$ for some $v \in \Omega^{j - 1} (M)$. Then
    \[ u = d (e^{t f} v) \in \text{Im} d. \]
    \item \textbf{$e^{- t f}$ is surjective:} let $u + \text{Im} d_t \in H^j (M)$.
    Consider $e^{t f} u + \text{Im} d$. Since
    \[ d (e^{t f} u) = e^{t f} (d_t u) = 0, \]
    $e^{t f} u + \text{Im} d \in H^j_t (M)$. Also, $e^{- t f}$ maps
    $e^{\text{tf}} u + \text{Im} d$ to $u + \text{Im} d_t$.
  \end{itemize}
\end{proof}

In particular, this theorem tells us
\[ \dim H_t^q (M) = \dim H^q(M) = b_q \]
for all $0\leq q\leq n$ and $t \geq 0$.

\subsection{Witten Laplacian}
Let 
\[d_t^{(q)*}:\Omega^{q+1}(M)\rightarrow\Omega^{q}(M)\]
be the formal adjoint of $d_t^{(q)}:\Omega^{q}(M)\rightarrow\Omega^{q+1}(M)$ with respect to $(\cdot,\cdot)$, i.e.,
\[
  (d_t^{(q)} u,v)=(u, d_t^{(q)*} v)
\]
for all $u\in \Omega^q(M)$ and $v\in \Omega^{q+1}(M)$. We call the operator
\[
  \triangle_t^{(q)}:=d_t^{(q)}d_t^{(q)*}+d_t^{(q)*}d_t^{(q)}:\Omega^q(M)\rightarrow\Omega^q(M)
\]
the \textbf{Witten Laplacian}. Note that this operator also depends on the Riemannian metric $g$.
   
It has the local form stated in the following theorem.
\begin{thm}[Bocher type formula]
  Let $p\in M$. Define
  \[
    \Lambda^\bullet T^*_pM:=\bigoplus_{k=0}^n\Lambda^k T^*_pM
  \]
  Then for each $\omega\in \Lambda^\bullet T^*_pM$,
  \begin{align}\label{E:Bochner}
    \triangle_t\omega=\triangle \omega + t|df|^2 \omega+t\sum_{k,l}\text{Hess}_f(\frac{\partial}{\partial x_l},\frac{\partial}{\partial x_k})[dx_l\wedge,dx_k\lrcorner]\omega.
  \end{align}
\end{thm}

\begin{xrem}
\begin{enumerate}[(i)]
  \item We define $|df|^2:=g(df,df)$ and 
  \begin{align}\label{def:Hess}
  \text{Hess}_f:=\nabla df\in \Gamma(T^* M\otimes T^*M).
  \end{align}
  The latter is said to be the \textbf{Hessian of $f$}. If $p\in M$ is a critical point of $f$, then $d^2f(p)=\text{Hess}_f(p)$ (see Definition 2.1 for the meaning of $d^2f$). A justification is given in the following: let $\{\frac{\partial}{\partial x_i}\}_{i=1}^n$ be a basis for $T_p M$. Then
  \[\text{Hess}_f(\frac{\partial}{\partial x_j},\frac{\partial}{\partial x_k})(p)=\frac{\partial^2 f}{\partial x_j\partial x_k}(p)-\frac{\partial f}{\partial x_l}(p)\Gamma_{jk}^l(p)=\frac{\partial^2 f}{\partial x_j\partial x_k}=d^2 f(\frac{\partial}{\partial x_j},\frac{\partial}{\partial x_k})(p)\]
  by the definition of critical point and formula (\ref{E:d^2f}). $[dx_l\wedge,dx_k\lrcorner]$ is simply a shorthand of $dx_l\wedge dx_k\lrcorner-dx_k\lrcorner dx_l\wedge$. As a result, there is no harm to called both $d^2 f$ and $\text{Hess}_f$ Hessian.
  \item By the local formula above and simple calculations, we notice that $\triangle^{(q)}_t$ is an positive elliptic operator with the symbol 
\[g^{ij}\xi_i\xi_j.\] 
It has an self-adjoint extension
\[\triangle^{(q)}_t: \text{Dom}\triangle_t^{(q)}\subseteq L^2(M)\rightarrow L^2(M),\]
which we use the same notation. The space $\text{Dom}\triangle_t^{(q)}$ is defined by
\[\text{Dom}\triangle_t^{(q)}=\{u\in L^2_q(M):\triangle_t^{(q)}u\in L^2_q(M)\}.\]

Following by \cite[theorem 8.3]{Shubin}, there exists a complete orthonormal system $\{f_j\}$ consisting of eigenfunctions of $\triangle_t$. Here, $f_j$ are all smooth, $\triangle_t f_j=\lambda_j f_j$ and $\{\lambda_j\}$ is a non-negative sequence tends to $+\infty$. Also, the spectrum $\sigma(\triangle_t)$ coincides with the set of all and eigenvalues. Note that for any $\lambda\in\mathbb{R}$, $\lambda\in\sigma(\triangle_t)$ if and only if the one of the three conditions holds:
\begin{enumerate}
  \item $\triangle_t-\lambda I$ is not injective (in this case, we call $\lambda$ an \textbf{eigenvalue}),
  \item $\triangle_t-\lambda I$ is not surjective,
  \item $\triangle_t-\lambda I$ is bijective, but it's inverse 
  \[(\triangle_t-\lambda I)^{-1}:L^2_q(M)\rightarrow L^2_q(M)\]
is not continuous (with respect to $\|\cdot\|$).
\end{enumerate}
\end{enumerate}
\end{xrem}

\begin{proof}[Proof of Bochner Type Formula]

Firstly, note that 
\begin{align}\label{E:dt}
  d_t=d+tdf\wedge
\end{align}
since for all $u\in \Omega^{\bullet}(M)$, 
\begin{align}
  d_t u = e^{-tf}d(e^{tf}u)=e^{-tf}e^{tf}t df\wedge u+e^{-tf}e^{tf}du=(d+tdf\wedge)u.
\end{align}

Next, formula (\ref{E:dt}) yields 
\begin{align}\label{E:dt*}
  d_t^*=d^*+tdf\lrcorner
\end{align}
because for all $u\in \Omega^{\bullet}(M)$ and $v\in \Omega^{\bullet+1}(M)$,
\begin{align}
\begin{split}
  (u,(d^*+tdf\lrcorner)v)_g =(u,d^*v)_g+(u,tdf\lrcorner v)=(du,v)+(tdf\wedge u,v) \\ 
  =((d+tdf\wedge)u,v)=(d_tu,v)=(u,d_t^*v).
\end{split}
\end{align}

  From formula (\ref{E:dt}) and (\ref{E:dt*}), we may derive that for all $u\in\Lambda^{\bullet}T_p^*M$,
  \begin{align}
  \begin{split}
    \triangle_t u=(dd^*+d^*d)u+t(df\wedge d^*u+d (df\lrcorner u)+df\lrcorner du+d^*(df\wedge u)) \\
  +t^2(df\wedge df\lrcorner+df\lrcorner df\wedge)u.
  \end{split}
  \end{align}

  Since the formula (\ref{E:Bochner}) is tensorial, we may verify the formula pointwise. Recall the exponential map 
  \[exp_p:B_0(r)\rightarrow M\]
is a diffeomorphism onto its image ($0$ is the origin of $\mathbb{R}^n$ and $r$ is a sufficiently small positive real number). It yields a local chart $(U,\varphi,x)$ for $M$ around $p$, called the \textbf{normal geodesic coordinate}. In the normal geodesic coordinate, the components of the Riemannian metric on $T_p M$ $g_{ij}(p)=g(\frac{\partial}{\partial x_i},\frac{\partial}{\partial x_j})$ simplify to $\delta_{ij}$. Thus, this implies 
\[
  \{dx_{i_1}\wedge\cdots\wedge dx_{i_k}:i_1<\cdots<i_k \text{ and }k=0,\ldots,n\}
\] 
is an orthonormal basis for $\Lambda^\bullet T^*_p M$.
  
Note that the both sides of formula (\ref{E:Bochner}) are invariant under coordinate change, it suffices for us to check it is true under normal coordinate. Furthermore, we only need to check those $\omega$ that are basis for $\Lambda^\bullet T^*_p M$. Write 
  \[df=\sum_{a=1}^n\frac{\partial f}{\partial x_a} dx_a\quad\text{and}\quad \]
where $dx_I=\omega=dx_{i_1}\wedge\cdots\wedge dx_{i_k}$ ($I=\{1\leq i_1<\cdots<i_k\leq n\}$). Then 
\begin{align}
\begin{split}
  (df\wedge df\lrcorner+df\lrcorner df\wedge)dx_I=\frac{\partial f}{\partial x_a} dx_a\wedge \frac{\partial f}{\partial x_b} \frac{\partial}{\partial x_b} dx_I+\frac{\partial f}{\partial x_a} \frac{\partial}{\partial x_a} \frac{\partial f}{\partial x_a} dx_a\wedge dx_I \\
  =\left(\frac{\partial f}{\partial x_a}\right)^2\left( (-1)^l\delta^a_Idx_a\wedge dx_{I-\{a\}}+\frac{\partial}{\partial x_a}(1-\delta^a_I) dx_a\wedge dx_I \right) \\
  =\left(\frac{\partial f}{\partial x_a}\right)^2 \left( (-1)^{2l}\delta^a_Idx_I+(1-\delta^a_I)dx_I \right)=\left(\frac{\partial f}{\partial x_a}\right)^2 dx_I=|df|^2 dx_I.
\end{split}
\end{align}
where $\delta^a_I=1$ if $a\in I$, $\delta^a_I=0$ otherwise. If $a\in I$, then $l$ is the integer so that $i_l=a$ and $dx_{I-\{a\}}$ equals to $dx_{i_1}\wedge\cdots\wedge \hat{dx_{i_l}}\wedge\cdots\wedge dx_{i_k}$.

Finally, also in the normal coordinate near $p$, we calculate the ''coefficient'' of $t$ in formula (\ref{E:Bochner}). Assume near $p$, $u$ can be written as
\[
  u=\sum_{i_1<\cdots<i_q} u_{i_1\cdots i_q} dx_{i_1}\wedge\cdots\wedge dx_{i_q}.
\]
Applying lemmas \ref{lem:d-local} and \ref{lem:d^*-local} we get
\begin{align}\label{E:3.8}
\begin{split}
  df\wedge d^*u=\sum_{i=1}^n \frac{\partial f}{\partial x_j}dx_j\wedge(-\sum_{i=1}^n dx_i\lrcorner\nabla_{\frac{\partial}{\partial x_i}}u) \\
  =-\sum_{i,j=1}^n\sum_{i_1<\cdots<i_q} \frac{\partial f}{\partial x_j}\frac{\partial u_{i_1\cdots i_q}}{\partial x_i} dx_j\wedge dx_i\lrcorner dx_{i_1}\wedge\cdots\wedge dx_{i_q},
\end{split}
\end{align}
\begin{align}\label{E:3.9}
\begin{split}
  df\lrcorner du=\sum_{i=1}^n \frac{\partial f}{\partial x_j}dx_j\lrcorner(\sum_{i=1}^n dx_i\wedge\nabla_{\frac{\partial}{\partial x_i}}u) \\
  =\sum_{i,j=1}^n\sum_{i_1<\cdots<i_q} \frac{\partial f}{\partial x_j}\frac{\partial u_{i_1\cdots i_q}}{\partial x_i} dx_j\lrcorner dx_i\wedge dx_{i_1}\wedge\cdots\wedge dx_{i_q},
\end{split}
\end{align}
\begin{align}\label{E:3.10}
\begin{split}
  ddf\lrcorner u=d\sum_{i=1}^n\sum_{i_1<\cdots< i_q} \frac{\partial f}{\partial x_i} u_{i_1\cdots i_q} dx_i\lrcorner dx_{i_1}\wedge\cdots\wedge dx_{i_q} \\
  \sum_{i,j=1}^n \frac{\partial^2 f}{\partial x_j \partial x_i} dx_j\wedge dx_i\lrcorner (\sum_{i_1<\cdots< i_q}u_{i_1\cdots i_q}dx_{i_1}\wedge\cdots\wedge dx_{i_q}) \\
  +\sum_{i,j=1}^n\sum_{i_1<\cdots< i_q} \frac{\partial f}{\partial x_j} \frac{\partial u_{i_1\cdots i_q}}{\partial x_i} dx_j\wedge dx_i\lrcorner dx_{i_1}\wedge\cdots\wedge dx_{i_q},
\end{split}
\end{align}
and
\begin{align}\label{E:3.11}
\begin{split}
  d^*df\wedge u=d\sum_{i=1}^n\sum_{i_1<\cdots< i_q} \frac{\partial f}{\partial x_i} u_{i_1\cdots i_q} dx_i\wedge dx_{i_1}\wedge\cdots\wedge dx_{i_q} \\
  -\sum_{i,j=1}^n \frac{\partial^2 f}{\partial x_j \partial x_i} dx_j\lrcorner dx_i\wedge (\sum_{i_1<\cdots< i_q}u_{i_1\cdots i_q}dx_{i_1}\wedge\cdots\wedge dx_{i_q}) \\
  -\sum_{i,j=1}^n\sum_{i_1<\cdots< i_q} \frac{\partial f}{\partial x_j} \frac{\partial u_{i_1\cdots i_q}}{\partial x_i} dx_j\lrcorner dx_i\wedge dx_{i_1}\wedge\cdots\wedge dx_{i_q}.
\end{split}
\end{align}

Summing up formulas (\ref{E:3.8}) to (\ref{E:3.11}), we get the coefficient of $t$ equals to
\begin{align}\label{E:3.12}
\begin{split}
  df\wedge d^*u+df\lrcorner du+ddf\lrcorner u+d^*df\wedge u \\
  =\sum_{i,j=1}^n\frac{\partial^2 f}{\partial x_i\partial x_j}dx_i\wedge dx_j\lrcorner u-\sum_{i,j=1}^n\frac{\partial^2 f}{\partial x_i\partial x_j}dx_i\lrcorner dx_j\wedge u \\
  =\sum_{i,j=1}^n\frac{\partial^2 f}{\partial x_i\partial x_j}dx_i\wedge dx_j\lrcorner u-\sum_{i,j=1}^n\frac{\partial^2 f}{\partial x_j\partial x_i}dx_i\lrcorner dx_j\wedge u \\
  = \sum_{i,j=1}^n\frac{\partial^2 f}{\partial x_i\partial x_j}[dx_i\wedge,dx_j\lrcorner]\wedge u\\
\end{split}  
\end{align}
because $f$ is smooth so that 
\[
  \frac{\partial^2 f}{\partial x_i\partial x_j}=\frac{\partial^2 f}{\partial x_j\partial x_i}
\]
for all $i,j\in\{1,\ldots,n\}$. The notation $[dx_i\wedge,dx_j\lrcorner]$ equals to $dx_i\wedge dx_j\lrcorner-dx_j\lrcorner dx_i\wedge$ for all $i,j\in\{1,\ldots,n\}$.

In the normal coordinate, we have $\Gamma_{ij}^k(p)=0$. Thus,
  \[\text{Hess}_f(\frac{\partial}{\partial x_i},\frac{\partial}{\partial x_j})=\frac{\partial^2 f}{\partial x_i\partial x_j}(p)-\frac{\partial f}{\partial x_k}(p)\Gamma_{ij}^k(p)=\frac{\partial f}{\partial x_i \partial x_j}
  \] 
and so the last line of formula (\ref{E:3.12}) equals to
\begin{align}
\begin{split}
  \sum_{i,j=1}^n \text{Hess}_f(\frac{\partial}{\partial x_i},\frac{\partial}{\partial x_j})[dx_i\wedge,dx_j\lrcorner]u.
\end{split}
\end{align}
This expression is invariant under change of coordinates, and so we complete the derivation of Bochner formula.
\end{proof}

Denote 
\[
  E^r_{\mu, t} (M)=\{u\in\Omega^r(M):\triangle_t^{(r)}u=\mu u\}
\]
as the eigenspace of $\mu$ with respect to $\triangle_t^{(r)}$ in $\Omega^r(M)$.
Note that the deformed exterior derivative 
\[d_t : E^r_{\mu, t} (M) \rightarrow E_{\mu, t}^{r + 1} (M)\]
sends $r$-forms from $E^r_{\mu, t} (M)$ to $E_{\mu, t}^{r + 1} (M)$, for
$r = 0, \ldots, n - 1$. Indeeds, let $u \in E_{\mu, t}^r (M)$, i.e.,
$\triangle_t^{(r)} u = \mu u$ for some $\mu \in \mathbb{R}$. Then we have
\[\triangle_t^{(r)} (d_t u)  =  (d_t d^*_t + d^*_t d_t) (d_t u)=  d_t d^*_t d_t u =  d_t d^*_t d_t u + d_t d_t d^*_t u =  d_t (\triangle_t^{ (r)} u) = \mu d_t u.\]

Let $I \subset \mathbb{R}$ be a bounded set and we denote
\[ \begin{array}{l}
     E_{I, t}^q (M) := \bigoplus_{\mu \in \text{Spec} \triangle^{(q)}_t \cap
     I} E_{\mu, t}^q (M) .
   \end{array} \]
\begin{thm}\label{thm:exact-eigenspace-sequence}
  The following sequence of vector spaces
  \[ E^0_{(0, \lambda], t} (M) \xrightarrow{d^{(0)}_t} E^1_{(0, \lambda], t}
     (M) \xrightarrow{d^{(1)}_t} E^2_{(0, \lambda], t} (M)
     \xrightarrow{d^{(2)}_t} \cdots \xrightarrow{d^{(n - 1)}_t} E^n_{(0,
     \lambda], t} (M) \xrightarrow{d^{(n)}_t} 0 \]
  is exact.
\end{thm}

\begin{proof}
  Fix $r = 0, \ldots, n$. We need to show $\ker d^{(r)}_t = \text{im} d^{(r -
  1)}_t .$
  \begin{itemize}
    \item [$(\supseteq)$:] Let $u \in \text{im} d^{(r - 1)}_t$. Then $u = d_t v$
    for some $v \in E^{r - 1}_{(0, \lambda], t} (M)$. Since $d^2_t = 0$, we
    must have $d_t u = d^2_t v = 0$. This gives $u \in \ker d^{(r)}_t$.
    
    \item [$(\subseteq)$:] Let $u \in \ker d^{(r)}_t$. We may write
    \[ u = u_1 + \cdots + u_N \]
    where $u_j \in E_{\mu_j, t}^r (M)$, $\mu_j$ distinct real positive numbers.
    Applying $d_t$ on both sides of the equation above, we get
    \[ 0 = d_t u = d_t u_1 + \cdots + d_t u_N . \]
    Since each $d_t u_j$ belongs to different eigenspace, $\{ d_t u_j \}$ is
    linearly independent. This implies $d_t u_j = 0$ for $j = 1, \ldots, N$.
    Fixing $j$, we have
    \[ u_j = \frac{1}{\mu_j} \triangle^{(r)}_t u_j  =  \frac{1}{\mu_j} (d_t d^*_t + d^*_t d_t) u_j = \frac{1}{\mu_j} d_t d^*_t u_j =  d_t \left( \frac{1}{\mu_j} d^*_t u_j \right) . \]
    Summing $u_j$ together, we get
    \[ u = d_t \left( \frac{1}{\mu_1} d^*_t u_1 + \cdots +
       \frac{1}{\mu_N} d^*_t u_N \right) \in \text{im} d^{(r-1)}_t . \qedhere\]
  \end{itemize} 
\end{proof}

Note that the spaces $E_{(0, \lambda], t}^q (M)$ are finite-dimensional for
all $q = 0, \ldots, n$ since each of them is a finite direct sum of
eigenspaces, and the eigenspaces of Witten Laplacian is finite-dimensional.
Using Theorem \ref{thm:exact-eigenspace-sequence} and dimension theorem for vector spaces, we then have
\[ \dim\text{im}\,d^{(q)}_t = \dim E^q_{(0, \lambda], t} (M) - \dim\ker d^{(q)}_t = \dim E^q_{(0, \lambda], t} (M) - \dim\text{im}\, d^{(q - 1)}_t \]
for $q = 1, \ldots, n$. For the case $q=0$, observe that $\ker d_t^(0)$ is a trivial vector space. For any $u\in \ker d_t^{(0)}$, we can write 
\[
  \alpha=\alpha_1+\cdots+\alpha_N
\]
where $\alpha_j\in E^0_{\mu_j,t}(M)$, $\mu_j$ distinct real positive numbers. Applying the Laplacian on both sides we get
\[
  \triangle_t^{(0)}\alpha=\mu_1\alpha_1+\cdots+\mu_N \alpha_N
\]
But since 
\[
  \triangle_t^{(0)}\alpha=d_t^{*(0)}d_t^{(0)}\alpha=0,
\] 
we must have $u_1=\cdots=u_N=0$ by linear independence of distinct eigenspaces. Thus, $\alpha=0$, i.e., $\ker d_t^{(0)}=0$.

and $\dim\text{im} d^{(0)}_t = \dim E_{(0, \lambda],
t}^0 (M)$. Then we may derive the following alternating sum
\begin{align}\label{E:3.13}
 \sum^q_{j = 0} (- 1)^{q - j} \dim E_{(0, \lambda], t}^q (M) = \dim\text{im}
   d^{(q)}_t \geq 0.
\end{align}   
Moreover, the equation holds when $q = n$, since $\text{rank} d^{(q)}$
vanishes as $q = n$. As a result, we obtain estimations encountering $\dim
H^j_t (M)$, and these estimations play important rules in the next section of
proving Morse inequalities.

\begin{thm}
  For any positive real number $\lambda>0$ and $q = 0, \ldots, n$, we have
  \begin{align}\label{E:3.15.2}
    \dim H^q_t (M) \leq \dim E^q_{[0, \lambda], t} (M)
  \end{align}
  and
  \begin{align}\label{E:3.15}
    \sum^q_{j = 0} (- 1)^{q - j} \dim H^j_t (M) \leq \sum^q_{j=0} (- 1)^{q - j} \dim E_{[0, \lambda], t}^q (M).
  \end{align}
  The equality of (\ref{E:3.15}) holds when $q=n$.
\end{thm}
\begin{proof} $ $
\begin{enumerate}
\item First, we proof that 
\[\dim H^q_t(M)=\dim \ker\triangle_t^{(q)}.\]
Let \{$v_1+\text{im}d_t^{(q-1)},\ldots,v_l+\text{im}d_t^{(q-1)}\}$ be the basis for $H^q_t(M)$. Fix $i=1,\ldots,l$. Then we can see
\begin{align*}
d_tv_i&=0 \\
(v_i,d_tu)&=0,\quad \forall u\in\Omega^{q-1}(M).
\end{align*}
As a consequence, for any $v\in\Omega^q(M)$,
\[(\triangle_t v_i,u)=(d_tv_i,d_tu)+(d_t^*v_i,d_t^*u)=0,\]
which means $v_i\in\ker\triangle_t^{(q)}$. Thus, $\dim H^q_t(M)\leq\dim \ker\triangle_t^{(q)}$. 

On the other hand, let $\{u_1,\ldots,u_{l'}\}$ be the basis for $\ker\triangle_t^{(q)}$. Fix $i=1,\ldots,l'$. For each $u\in \Omega^q(M)$,
\[
  0=(\triangle_t v_i,v_i)=(d_tv_i,d_tv_i)+(d^*_tv_i,d^*_tv_i),
\]
which implies that $d_tv_i=0$ and $d^*_tv_i=0$. As a result, $v_i+\text{im}d_t^{q-1}\in H^q_t(M)$, i.e., $\dim H^q_t(M)\geq\dim \ker\triangle_t^{(q)}$.
\item By the first part of proof, we can immediately prove (\ref{E:3.15.2}) since 
\begin{align*}
\dim E^q_{[0,\lambda],t}(M)=\dim(\ker\triangle_t^{(q)}\oplus E^q_{(0,\lambda],t}(M))=\dim\ker\triangle_t^{(q)}+E^q_{(0,\lambda],t}(M) \\
\geq \dim\ker\triangle_t^{(q)}=\dim H^q_t(M).
\end{align*}
\item Part 1 of this proof and inequality (\ref{E:3.13}) says that
\[\dim H^q_t(M)=\dim\ker\triangle_t^{(q)}\]
and
\[\sum^q_{j = 0} (- 1)^{q - j} \dim E_{(0, \lambda], t}^q (M)\geq 0\]
respectively. Then
\begin{align*}
\sum^q_{j=0} (- 1)^{q - j} \dim E_{[0, \lambda], t}^q (M)=\sum^q_{j=0} (- 1)^{q - j} (\dim E_{(0,\lambda],t}^q(M)+\dim\ker\triangle_t^{(q)}) \\
=\sum^q_{j=0} (- 1)^{q - j} \dim E_{(0, \lambda], t}^q (M)+\sum_{i=1}^q(- 1)^{q - j}\dim H^q_t(M) \\
\geq \sum_{i=1}^q(- 1)^{q - j}\dim H^q_t(M),
\end{align*}
ending the proof of formula (\ref{E:3.15}).
\end{enumerate}
\end{proof}

In the next section, we will prove the following theorem. If the theorem hold, plug the previous theorem and then we derive the Morse inequalities.

\begin{thm}
  There exist $C > 0$ and $t > 0$ such that when $t\geq t_0$,
  \[ \dim E^q_{[0,e^{- C t}], t} = m_q \]
  for all $q = 0, \ldots, n$.
\end{thm}

\subsection{Locally Flat Metric Near Critical Points}

Let $\tilde{g}$ be arbitrary Riemannian metric on $M$. Fix $p\in \text{Crit}(f)$. Let $(U_p,\phi)$ be the coordinate chart near $p$ given by Morse lemma (lemma \ref{lem:Morse}). We may assume $U_p\cap U_{p'}=\emptyset$ when $p\neq {p'}$. Because $\phi$ is a local diffeomorphism from $U$ to $\phi(U_p):=\tilde{U_p}$, we can define a Riemannian metric $g_p$ on $U$ by
\[g_p(u,v)(\hat{p})=\langle d\phi_{\hat{p}}(u),d\phi_{\hat{p}}(v)\rangle\]
for each $\hat{p}\in U$, $u,v\in T_{\hat{p}}M$, and  $\langle\cdot,\cdot\rangle$ denotes the standard Euclidean metric on $\mathbb{R}^n$. 

Denote $V=M-\cup_{p\in \text{Crit}(f)} U_p$. Then 
\begin{align}\label{E:OpenCovering}
S=\{U_p\}_{p\in \text{Crit(f)}}\cup\{V\}
\end{align}
is an open covering of $M$. Let $\{\varphi_p\}\cup\{\psi\}$ be a partition unity subordinate on $S$, where 
\begin{align*}
  \text{supp}\varphi_p &\subseteq U_p \\
  \text{supp}\psi &\subseteq V
\end{align*}
We define a new metric $g$ by 
\begin{align}\label{E:SpecialMetric}
  g=\tilde{g}\psi+\sum_{p\in \text{Crit(f)}}g_p\varphi_p
\end{align}

From now on, we fix $g$ to be the metric defined by (\ref{E:SpecialMetric}).

\subsection{Witten Laplacian on $\mathbb{R}^n$}

Let $p\in \text{Crit}(f;r)$. Let $(U,\varphi,x)$ be local chart $p$ given by Morse lemma \ref{lem:Morse}. Let $u\in C^\infty_0 (U)$ and $J=\{1\leq j_1<\cdots<j_q\leq n\}$. Then by our construction of $g$, the Bochner type formula of $\triangle_t$ on $U$ becomes
\begin{align}
  \triangle_t (u dx_J)=\left[-\triangle u(x)+t^2 |x|^2 u(x)\right] dx_J+ t u(x)\sum_{i=1}^n \varepsilon_j[dx_j\wedge,dx_j\lrcorner](dx_J).
\end{align}
where $dx_J=dx_{j_1}\wedge\cdots\wedge dx_{j_q}$,
\[
  \triangle = \frac{\partial^2}{\partial x_1^2}+\cdots+\frac{\partial^2}{\partial x_n^2} \qquad |x|^2=x_1^2+\cdots+x_n^2\] 
and 
\[
  \varepsilon_j=
  \begin{cases}
    -1   & \text{for }j\leq r \\
    1    & \text{for }j>r \\
  \end{cases}
\]

\section{A study on eigenspaces}

In this section, we will study the eigenspaces of $\triangle_t^{(q)}$. We will show that when $t$ is sufficiently large, there is exactly $m_q$ eigenvalues that are bounded by $e^{-Ct}$. Once we establish this, the proof of main theorem is complete by the reduction we performed last section.

\subsection{Norms on $\Omega^q(M)$}
For any $\omega\in \Omega^q(M)$, define
\begin{align}
  \|\omega\|_{\infty}=\sup_{p\in M}|\omega(p)|,
\end{align}
called the \textbf{sup norm of $\omega$} (we write $|\cdot|$ for $\sqrt{g(\cdot,\cdot)}$). Note that there is another way to define the sup norm of differential form. Let $\{(U_i,\varphi_i,x)\}_{i=1}^n$ be a set of charts that covers $M$. Assume further that the exterior bundle $\Lambda^q T^*M$ trivializes on each $U_i$ with the trivialization (composited with $\varphi_i\times id$):
\[\eta_i:\Lambda^q T^*M|_{U_i}\xrightarrow{\sim} \varphi_i^{-1}(U_i)\times\mathbb{R}^{\binom{n}{q}}.\]
We can define the induced map $\eta_{i*}:\Gamma(U_i,E|_{U_i})\rightarrow\Gamma(\varphi^{-1}(U_i),\mathbb{R}^{\binom{n}{q}})$ by
\[\eta_{i*}\xi=\eta_i\circ\xi\circ\varphi_i^{-1}.\]
Finally, let $\{\rho_i\}_{i=1}^N$ be a partition of unity subordinated to $\{U_i\}$.

With the notations prepared, we can then define the sup norm of $\omega$ by
\[\omega=\sum_{i=1}^{\binom{n}{q}}\sum_{i_1<\cdots< i_q}\|\omega_{i,i_1,\ldots,i_q}\|_{\infty,\mathbb{R}}\]
where $\|\cdot\|_{\infty,\mathbb{R}}$ is the sup norm on $\mathbb{R}$.

With a similar fashion, we can define the Sobolev $s$ norm on $\Omega^q(M)$ by
\[\|\omega\|_s=\sum_{i=1}^{\binom{n}{q}}\sum_{i_1<\cdots< i_q}\|\omega_{i,i_1,\ldots,i_q}\|_{s,\mathbb{R}}\]
where 
\[\eta_{i*}(\omega\varphi_i)=\sum_{i_1<\ldots<i_q}\omega_{i,i_1,\ldots,i_q}dx_{i_1}\wedge\cdots\wedge dx_{i_q}\]
and $\|\cdot\|$ is the Sobolev $s$ norm on $\mathbb{R}$ defined by 
\[\|f\|_{s,\mathbb{R}}^2=\sum_{|\alpha|\leq 2n} \|\partial^\alpha h\|^2_{L^2(\mathbb{R}^n)}\]

where $|\alpha|=\alpha_1+\cdots+\alpha_n$ and $\partial^{\alpha}=\frac{\partial^{\alpha_1}}{\partial x_1^{\alpha_1}}\cdots\frac{\partial^{\alpha_n}}{\partial x_n^{\alpha_n}}$.
is the $s$ norm on $\mathbb{R}$.

\subsection{Lower bound of $\dim E^q_{[0,e^{-Ct}],t}(M)$}

Let $p \in \text{Crit} (f, q)$ and $(U_p,\varphi, x)$ be a local chart such that $x(p) = 0$. Let $\kappa\in C^\infty_0(\mathbb{R})$ be a smooth function with compact support so that 
\begin{enumerate}
  \item $\text{supp}\kappa\subseteq [-2,2]$
  \item $\kappa(y)=1$ for all $x\in [-1,1]$
  \item $0\leq\kappa(y)\leq 1$ for all $x\in \mathbb{R}$.
\end{enumerate}
For any $\varepsilon >0$, set $\kappa_\varepsilon(y)=\kappa(\varepsilon^{-1}y)$. We may choose $\varepsilon$ small enough such that the cube centered at $0$ is contained in $\varphi(U_p)$ ,i.e.,
\[C_\varepsilon(0):=\{x\in\mathbb{R}:\sup_{1\leq i\leq n}|x_i|\leq 2\varepsilon\}\subseteq\varphi(U_p)\]

Now, define $\chi(x)=\kappa_\varepsilon(x_1)\cdots\kappa_\varepsilon(x_n)$ such that $\chi = 1$ near $p$.
Let $v \in \Omega_0^q (U_p) \subset \Omega^q (M)$, where $\Omega_0^q (U_p)$ means the set of
smooth $q$-forms with compact support in $U_p$ defined by
\begin{align}\label{E:localexponential}
  v\circ\varphi^{-1} (x) = e^{- \frac{t}{2} | x |^2} \chi (x) d x^1 \wedge d x^2 \wedge \cdots
   \wedge d x^q .
\end{align}
The following lemma gives a upper bound of the $L^2$ norm of
$(\triangle_t^{(q)})^j v$

\begin{lem}\label{lem:LaplacianBound}
  For all $s \in \mathbb{Z}^+$, there exist $t_s > 0$, depending on $s$, such that when $t \geq t_s$, we have
  \[ \| (\triangle_t^{(q)})^s v \| \leq \exp(\frac{-t\varepsilon^2}{4}) . \]
\end{lem}

\begin{proof}
  By calculation
  \begin{align}
    \triangle_t^{(q)}v=\triangle_t^{(q)} (\chi(x)\exp(-\frac{t|x|^2}{2}))
  =\exp(-\frac{t|x|^2}{2})\sum_{i=1}^n(2tx_i\frac{\partial\kappa_\varepsilon(x_i)}{\partial x_i}-\frac{\partial^2\kappa_\varepsilon(x_i)}{\partial x_i^2})\prod_{j\neq i}\kappa_\varepsilon(x_j).
  \end{align}
  Define 
  \[\tau_1(x_i)=2tx_i\frac{\partial\kappa_\varepsilon(x_i)}{\partial x_i}-\frac{\partial^2\kappa_\varepsilon(x_i)}{\partial x_i^2}\]
  Note that $\tau_1$ is non-negative the support of $\tau_1$ is contained in $[-2\varepsilon,-\varepsilon]\cup[\varepsilon,2\varepsilon]$. By an induction argument, one can show that 
  \begin{align}\label{E:4.3}
  (\triangle_t^{(q)})^kv=(\triangle_t^{(q)})^k (\chi(x)\exp(-\frac{t|x|^2}{2}))=\exp(-\frac{t|x|^2}{2})\sum_{i=1}^n\tau_k(x_i)\prod_{j\neq i}\kappa_\varepsilon(x_j)
  \end{align}
  where $\text{supp}\tau_k\subseteq [-2\varepsilon,-\varepsilon]\cup[\varepsilon,2\varepsilon]$.
  Now, by self-adjointness of $\triangle^{(q)}_t$ and formula (\ref{E:4.3}),
  \begin{align*}
    \|(\triangle^{(q)}_t)^sv\|^2&=((\triangle^{(q)}_t)^sv,(\triangle^{(q)}_t)^sv)=((\triangle^{(q)}_t)^{2s}v,v) \\
    &=n\int_\mathbb{R}\exp(-ty^2)\tau_{2s}(y)\kappa_\varepsilon(y)dy\left(\int_\mathbb{R}\exp(-ty^2)\kappa_\varepsilon(y)^2dy\right)^{n-1} \\
    &\leq 2n\sup_{y\in\mathbb{R}} \tau_{2s}(y)\int_\varepsilon^\infty \exp(-ty^2)(y)dy \left(\int_\mathbb{R}\exp(-ty^2)dy\right)^{n-1} \\
    &\leq 2n\sup_{y\in\mathbb{R}}\tau_{2s}(y)\frac{1}{t\varepsilon}\exp(-\frac{t\varepsilon^2}{2})\left(\frac{\pi}{t}\right)^{(n-1)/2}.
  \end{align*}
  As a result, we have 
  \[\|(\triangle_t^{(q)})^sv\|\leq \exp(-\frac{t\varepsilon^2}{4})\]
  when $t\geq t_2$ where $t_{s}=(2n\pi^{(n-1)/2}\sup_{y\in\mathbb{R}}\tau_{2s}(y))^{2/(n+1)}$ depending on $s$.
\end{proof}

There is another inequalities which is needed. 

\begin{lem}[Sobolev Type Inequality]\label{lem:Sobolev}
Let $h\in \Omega^q(M)$ be a smooth $q$-form on $M$. Then we have the following estimation
\begin{align}
  \|h\|_{\infty}\leq C\|h\|_{2n}
\end{align}
for some $C>0$.
\end{lem}
\begin{proof}
  The case on $\mathbb{R}^n$: Let $h \in \mathscr{S}(\mathbb{R}^n)$ (Schwartz class on $\mathbb{R}^n$). For all $x \in \mathbb{R}^n$,
\begin{align}
     | h (x) | & = \left| (2 \pi)^{- n} \int_{\mathbb{R}^n} e^{i \langle x,
     \xi \rangle} \hat{h} (\xi) d \xi \right| \leq (2 \pi)^{- n} \int |
     \hat{h} (\xi) | d \xi\\
     & =  (2 \pi)^{- n} \int_{\mathbb{R}^n} | \hat{h} (\xi) | (1 + | \xi
     |^2)^{n} (1 + | \xi |^2)^{-n} d \xi\\
     & \leq (2 \pi)^{- n} \left( \int_{\mathbb{R}^n} (1 + | \xi
     |^2)^{2 n} | \hat{h} (\xi) |^2 d \xi \right)^{1 / 2} \left(
     \int_{\mathbb{R}^n} (1 + | \xi |^2)^{- 2 n} d \xi \right)^{1 / 2} \label{E:4.7}
\end{align}
The second line is a common trick to derive upper bound envolving $L^2$ norm,
and the last inequalities comes from Cauchy-Schwartz.

Note that $\int_{\mathbb{R}^n} (1 + | \xi |^2)^{- 2 n} d \xi$ is bounded. Thus, it remains to estimate the integration of $(1+|\xi|)^{2n}\hat{h}(\xi)$:
\begin{align}
 \int_{\mathbb{R}^n}(1 + | \xi |^2)^{2 n} | \hat{h} (\xi) |^2 = \int_{\mathbb{R}^n}\sum_{\alpha_1+\cdots\alpha_n\leq 2n} \xi_1^{2\alpha_1}\cdots\xi_n^{2\alpha_n}\hat{h}(\xi)^2 d\xi \\
 =  \sum_{\alpha_1+\cdots\alpha_n\leq 2n}\int_{\mathbb{R}^n} \xi_1^{2\alpha_1}\cdots\xi_n^{2\alpha_n}\hat{h}(\xi)^2 d\xi \label{E:4.9}
\end{align}
Note that the Fourier transform of $\xi_1^{\alpha_1}\cdots\xi_n^{\alpha_n}\hat{h}(\xi)$ is 
\begin{align}
(-i)^{\alpha_1+\cdots+\alpha_n}\frac{\partial^{\alpha_1}}{\partial x_1^{\alpha_1}}\cdots\frac{\partial^{\alpha_n}}{\partial x_n^{\alpha_n}}h(x). \label{E:4.10}
\end{align}
Applying (\ref{E:4.9}), (\ref{E:4.10}) and Parseval's formula (cf. \cite[Theorem 7.1.6]{Hormander}), we then have
\begin{align}
\int_{\mathbb{R}^n}(1 + | \xi |^2)^{2 n} | \hat{h} (\xi) |^2 = (2\pi)^n\sum_{\alpha_1+\cdots+\alpha_n\leq 2n} \int_{\mathbb{R}^n} (\frac{\partial^{\alpha_1}}{\partial x_1^{\alpha_1}}\cdots\frac{\partial^{\alpha_n}}{\partial x_n^{\alpha_n}}h(x) )^2 dx \\
= (2\pi)^n\sum_{|\alpha|\leq 2n} \|\partial^\alpha h\|^2_{L^2(\mathbb{R}^n)}=(2\pi)^n\|h\|_{2n}^2, \label{E:L^2-Bound}
\end{align}
where $|\alpha|=\alpha_1+\cdots+\alpha_n$ and $\partial^{\alpha}=\frac{\partial^{\alpha_1}}{\partial x_1^{\alpha_1}}\cdots\frac{\partial^{\alpha_n}}{\partial x_n^{\alpha_n}}$.

Therefore, from (\ref{E:4.7}) and (\ref{E:L^2-Bound}), we derive that for all $x \in \mathbb{R}^n$
\begin{align} | h (x) | \leq C\|h\|_{2n}, 
\end{align}
where $C=(2\pi)^{-n/2}(\int_{\mathbb{R}^n}(1+|\xi|^2)^{-2n}d\xi)^{1/2}$, independent on $x$. So by taking supremum over $x \in
\mathbb{R}^n$ we prove that
\begin{align}
  \| h \|_{\infty} = \sup_{x \in \mathbb{R}^n} | h (x) | \leq C \| h
   \|_{2n}.
\end{align}

The case on closed manifold: let $\{(U_i,\varphi_i,x)\}_{i=1}^N$ be a collection of charts covers $M$. By passing to a finer cover, we may assume $\Omega^q(M)$ trivializes on each $U_i$ with the corresponding trivialization  (composited with $\varphi_i\times id$)
\[f_i:E|_{U_i}\rightarrow \varphi^{-1}(U_i)\times \mathbb{R}^{\text{rank}\Omega^q(M)}\]
By the previous case of the estimation on $\mathbb{R}^n$,
\begin{align*}
  \|h\|_{2n}'=\sum_{i=1}^N \|f_{i*}(\eta_ih)\|_{2n,\mathbb{R}^n}\leq\sum_{i=1}^N C_i\|f_{i*}\eta_ih\|_{\infty,\mathbb{R}^n} \leq \max_{1\leq i\leq N}\{C_i\}\sum_{i=1}^N \|f_{i*}\eta_ih\|_{\infty,\mathbb{R}^n} \\
  = \max_{1\leq i\leq N}\{C_i\}\|h\|'_{\infty}
\end{align*}
and this completes the proof. 
\end{proof}

Now fix $0 < C < \varepsilon^2/4$. Denote $L_q^2 (M)$ be the completion of $\Omega^q (M)$ with respect to
$(\cdot, \cdot)$. Let
\[ P_{\leq e^{- C t}}^q : L_q^2 (M) \rightarrow E_{[0, e^{- C t}], t}^q
   (M) \]
be the orthogonal projection with respect to $(\cdot, \cdot)$. Denote
$\tilde{v} = P_{\leq e^{- C t}}^q v$ be the orthogonal projection of $v$
onto $E_{[0, e^{- C t}], t}^q (M)$. We give an upper bound of the sup norm of
\[ 
\| v - \tilde{v} \|_{\infty} = \sup_{x \in M} | v (x) - \tilde{v} (x) |
\]
in the following theorem:

\begin{thm}\label{thm:4.3}
  We can find $\varepsilon_0 > 0$ and $t_0 > 0$ such that when $t \geqslant
  t_0$, we have
  \[ \| v - \tilde{v} \|_{\infty} \leq e^{- \varepsilon_0 t} . \]
\end{thm}

\begin{proof}
  Applying the G{\aa}rding inequalities \cite[theorem 8.1]{Taylor}, we get
\begin{align}\label{E:2nBound}
\|v-\tilde{v}\|_{2n}\leq \hat{C}t ((\triangle_t^{(q)})^{2n} (v-\tilde{v}), v-\tilde{v})+\hat{C}\|v-\tilde{v}\|=\hat{C}t\|(\triangle_t^{(q)})^n(v-\tilde{v})\|+\hat{C}\|v-\tilde{v}\|
\end{align}
Note that if we write $v=\sum_{i=0}^\infty u_i$, $u_i \in E_{\lambda_i,t}(M)$, then
\begin{align}
\|(\triangle_t^{(q)})^n(v-\tilde{v})\|=\sum_{i=N+1}^\infty \lambda^n_i \|u_i\|\leq \sum_{i=0}^\infty \lambda^n_i \|u_i\|=\|(\triangle_t^{(q)})^nv\|\leq \exp(-\frac{t\varepsilon^2}{4}),
\end{align}
for $t\geq t_n$ by lemma \ref{lem:LaplacianBound}. The number $N$ is the largest integer so that $\lambda_i\leq e^{-Ct}$ for all $i=0,\ldots, N$. 

On the other hand, also by lemma \ref{lem:LaplacianBound}, when $t\geq t_1$
\begin{align}
 \exp(-\frac{t\varepsilon^2}{4})\geq\|(\triangle_t^{(q)})v\|\geq\sum_{i=N+1}^\infty \lambda_i \|u_i\|\geq e^{-Ct} \sum_{i=N+1}^\infty \|u_i\|=e^{-Ct}\|v-\tilde{v}.\| 
\end{align}

Then we have
\begin{align}
\|v-\tilde{v}\|_{2n}\leq\hat{C}t\exp(-\frac{t\varepsilon^2}{4})+\hat{C}\exp(-(\frac{\varepsilon^2}{4}-C)t)
\end{align}
for $t\geq\max\{t_1,t_n\}$. Take $\varepsilon_0=\min\{\varepsilon^2/4,\varepsilon^2/4-C\}/2$. Then there exists $t_0\geq\max\{t_1,t_n\}$ so that when $t\geq t_0$, 
\[\|v-\tilde{v}\|_{2n}\leq \exp(-\varepsilon_0 t).\]
And by lemma \ref{lem:Sobolev},
\[\|v-\tilde{v}\|_\infty\leq \|v-\tilde{v}\|_{2n}\leq \exp(-\varepsilon_0 t)\]
when $t\geq t_0$.
\end{proof}
  Write $\text{Crit}(f;q)=\{p_1,\ldots,p_{m_q}\}$. By the theorem above we know that there exists a sequence of sections $\beta=\{\tilde{v}_1,\ldots,\tilde{v}_{m_q}\}$ so that $\tilde{v}_j\in E_{[0,e^{-Ct}],t}^q(M)$ and
\begin{align}
  \tilde{v}_j=v_j+R_j
\end{align}
where $v_j$ satisfies formula (\ref{E:localexponential}) and $\|R_j\|_\infty\leq e^{-\varepsilon_0t}$. 

To end this subsection, We claim that $\beta$ is linearly independent when $t$ is sufficiently large. We define $A(t),\tilde{A(t)}\in M_{m_q\times m_q}(\mathbb{R})$ by
\[
  A(t)_{ij}=(v_i,v_j)
\]
and
\[
  \tilde{A}(t)_{ij}=(\tilde{v_i},\tilde{v_j}).
\] 
By our construction, $A(t)$ is always positive definite diagonal matrix for any $t>0$. We check it in the following: for $i\neq j$, the supports of $v_i$ and $v_j$ are disjoint, and thus $(v_i,v_j)=0$. On the other hand, when $i=j$, $(v_i,v_i)>0$ since $v_i$ is non-trivial.

In particular, $\det A(t)>0$ for any $t>0$. Since 
\[\det:M_{m_q\times m_q}(\mathbb{R})\rightarrow \mathbb{R}\] 
is a real-valued continuous function, there exists $\delta>0$ so that $\det B>\det A(t)/2$ for any $B\in M_{m_q\times m_q}(\mathbb{R})$ with $\sup_{ij}|A(t)_{ij}-B_{ij}|<\delta$. But by Theorem \ref{thm:4.3},
\[|A(t)_{ij}-\tilde{A}_{ij}|\leq (\|R_i\|_{\infty}^2(v_j,v_j)+\|R_j\|_{\infty}^2(v_i,v_i)+\|R_i\|_{\infty}\|R_j\|_{\infty})\int_Mdv_g<\delta\]
when $t$ is sufficiently large. Thus, when $t$ is sufficiently large
\[(a_1\tilde{v}_1+\cdots+a_{m_q}\tilde{v}_{m_q},a_1\tilde{v}_1+\cdots+a_{m_q}\tilde{v}_{m_q})=\sum_{i,j=1}^{m_q}a_ia_j(\tilde{v}_i,\tilde{v}_j)\]
is identically zero if and only $a_1=\cdots=a_{m_q}=0$. Thus, there exists at least $m_q$ linear independent elements in $E^q_{[0,e^{-Ct}],t}(M)$, in other words,
\[\dim E^q_{[0,e^{-Ct}],t}(M)\geq m_q.\]

\subsection{Upper bound of $\dim E^q_{[0,e^{-Ct}],t}(M)$}

Conversely, we need to show $\dim E^q_{[0,e^{-Ct},t]}(M)\leq m_q$. To do this, it suffices to show $\beta$ we chose in the last section spans all elements in $E^q_{[0,e^{-Ct},t]}(M)$. We claim that 
\begin{thm}\label{thm:KeyEstimation}
  There exists $C>0$ and $s>0$ so that when $t\geq s$, for all $u\in \Omega^q(M)$, $u\in (\text{span}\beta)^{\perp}$, we have
  \[(\triangle_t^{(q)}u,u)\geq Ct\|u\|^2.\]
\end{thm}

Let $S$ be the open covering appeared in (\ref{E:OpenCovering}). Let $\{\varphi_U\}$ be a partition of unity subordinated on $S$ so that 
\[\sum_{U\in S} \varphi_U^2=1\]
and
\[\varphi_{U_p}=1 \text{ on } \text{supp}\chi\]
for all $p\in \text{Crit}(f)$. 

\begin{lem}\label{lem:1}
 There exists a constant $C>0$, independent of $t$, so that for all $t\geq 0$ and for all $u\in\Omega^q(M)$, 
 \[(\triangle^{(q)}_t u, u)\geq \sum_{U\in S}(\triangle_t^{(q)}(\varphi_U u),\varphi_U u)-C \|u\|^2.\]
\end{lem}

\begin{proof}
Applying formula (\ref{E:Bochner}), we have 
\[(\triangle_t^{(q)}u,u)=\int_M\left(|du|^2+|d^*u|^2+t^2|df|^2|u|^2+t(\sum_{k,l}\text{Hess}_f(\frac{\partial}{\partial x_l},\frac{\partial}{\partial x_k})[dx_l\wedge,dx_k\lrcorner]u,u)\right)\]
Since $\sum\varphi_Ud\varphi_U=0$ and 
\begin{align*}
  \sum|d(\varphi_Uu)|^2&=|du^2|+\sum|d\varphi_U\wedge u|^2 \\
  \sum|d^*(\varphi_Uu)|^2&=|d^*u^2|+\sum|d\varphi_U\lrcorner u|^2 
\end{align*}
we obtain for any $u\in \Omega^q(M)$, 
\[\sum_{U\in S}(\triangle_t^{(q)}(\varphi_U u),\varphi_U u)=\sum\int_M |d\varphi_U \wedge u|^2+|d\varphi_U\lrcorner u|^2 dv_M+(\triangle_t^{(q)} u,u)\leq C\|u\|^2+(\triangle_t^{(q)} u,u).\]
\end{proof}

\begin{lem}
  There exists $C>0$ and $s>0$ such that for all $u\in\Omega^q_0(V)$,
  \begin{align}
    (\triangle_t^{(q)} u,u)\geq C t\|u\|^2,
  \end{align}
  when $t\geq s$.
\end{lem}
\begin{proof}
  From the Bochner type formula, we know that 
  \[(\triangle_t u,u)=(\triangle u,u)+ t^2|df|^2\|u\|^2+t(\sum_{l,k}\text{Hess}_f(\frac{\partial}{\partial x_l},\frac{\partial}{\partial x_k})[dx^l\wedge,dx^k\lrcorner]u,u).\]
  On $V$, $|df|^2>c$ for some $c>0$, so there exists $t_0>0$ such that 
  \begin{align}
    (\triangle_t u,u)\geq \frac{c}{2} t^2 \|u\|^2 \geq \frac{ct_0}{2}t\|u\|^2
  \end{align}
  when $t\geq t_0$. Take $C=ct_0/2$ and $s=t_0$ and we complete the proof.
\end{proof}

In section 5, we will introduce a operator $\Box_t$ which can be viewed as the Witten Laplacian on $\mathbb{R}$ (see (\ref{E:Boxt})), and the spectrum of $\Box_t$ will be fully characterized. This result is essential in the proof of Lemma 4.7 and 4.8 below.

\begin{lem}
  For any $p\in\text{Crit}(f)\backslash \text{Crit}(f;l)$ there exists $C>0$ such that 
  \[(\triangle_t^{(l)}u,u)\geq Ct\|u\|^2\]
for any $u\in \Omega^l_0(U_p)$.
\end{lem}

\begin{proof}
  Since in $U_p$ the metric is flat and $|df|^2=|x|^2$, we may regard $u$ as a $l$ form on $\mathbb{R}^n$ and $\triangle_t$ as $\Box_t$. By our study of the spectrum of $\Box_t$, 
  $(\triangle_t u,u)\geq Ct\|u\|^2$ for some $C>0$.
\end{proof}

\begin{lem}
  For any $p\in \text{Crit}(f;q)$, there exist $C>0$, $s>0$ and $\varepsilon>0$ so that when $t\geq s$, 
  \[(\triangle_t^{(q)}(\varphi_{U_p}u),\varphi_{U_p}u)\geq C\left(t\|\varphi_{U_p}u\|^2-e^{-\varepsilon t}\|u\|^2\right)\]
for all $\varphi_{U_p}u\in \Omega^q(U_p)$ and $u\in (\text{span}\beta)^\perp$.
\end{lem}
\begin{proof}
  Like the previous proof, we also regard $u$ as a $q$ form on $\mathbb{R}^n$ and $\triangle_t$ as $\Box_t$. By our calculation in the appendix, there exists a complete orthonormal system $\{u_i\}_{i=0}^\infty$ with 
  \begin{align}
    \begin{cases}
    u_0=\pi^{-n/4}t^{-n/4}\exp(-tx^2/2)dx_1\wedge\cdots\wedge dx_q\in\ker{\Box_t} \\
    u_i\in E_{\lambda_i}(\Box_t)
    \end{cases}
  \end{align}
  also, $0<\lambda_1\leq\lambda_2\leq$ and $\lambda_i\sim O(t)$. Write $\varphi_{U_p}u=\sum_{i=0}^\infty a_i u_i$. Then 
  \begin{align*}
  (\Box_t (\varphi_{U_p}u),\varphi_{U_p}u)_{\mathbb{R}^n}&=\sum_{i=1}^n \lambda_ia_i^2\geq Ct \sum_{i=1}^n a_i^2=Ct(\|\varphi_{U_p}u\|^2-a_0^2) \\ &=Ct(\|\varphi_{U_p}u\|^2-t^{n/2}(\varphi_{U_p}u,e^{-tx^2/2}dx_1\wedge\cdots\wedge dx_q)^2).
  \end{align*}
  Triangle inequality gives 
  \begin{align*}
    |(\varphi_{U_p}u,e^{-tx^2/2}dx_1\wedge\cdots\wedge dx_q)|&\leq |(\varphi_{U_p}u,\chi e^{-tx^2/2}dx_1\wedge\cdots\wedge dx_q)| \\
    &+|(\varphi_{U_p}u,(1-\chi)e^{-tx^2/2}dx_1\wedge\cdots\wedge dx_q)|
  \end{align*}
  Cauchy-Schwartz inequality tells us
  \begin{align*}
    |(\varphi_{U_p}u,(1-\chi) e^{-tx^2/2}dx_1\wedge\cdots\wedge dx_q)|\leq \|\varphi_{U_p}u\|\|(1-\chi) e^{-tx^2/2}dx_1\wedge\cdots\wedge dx_q\| \\
    \leq Cte^{-t\varepsilon^2/4}\|\varphi_{U_p}u\|
  \end{align*}
  Also, we have
  \begin{align}
    |(\varphi_{U_p}u,\chi e^{-tx^2/2}dx_1\wedge\cdots\wedge dx_q)|=|(u,\tilde{v}-R_p)|=|(u,R_p)|\leq \|u\|e^{-\epsilon_0 t}
  \end{align}
  Therefore,
  $(\triangle_t^{(q)}(\varphi_{U_p}u),\varphi_{U_p}u)\geq \tilde{C}t(\|\varphi_{U_p}u\|^2-e^{\epsilon_0 t}\|u\|)$
  when $t$ is sufficiently large.
\end{proof}
\begin{proof}[Proof of Theorem \ref{thm:KeyEstimation}]
By Lemma \ref{lem:1}, 
\[(\triangle_t^{(q)}u,u)\geq\sum_{U\in S} (\triangle_t^{(q)} \varphi_{U_p}u,\varphi_{U_p}u)-C_1\|u\|^2\]
We have
\[\sum_{U\in S} (\triangle_t^{(q)} (\varphi_{U_p}u),\varphi_{U_p}u)=(\triangle_t^{(q)}(\varphi_{V}u),\varphi_{V}u)+\sum_{p\in\text{Crit}(f)\backslash\text{Crit}(f;q)}(\triangle_t^{(q)}(\varphi_{U_p}u),\varphi_{U_p}u)+\sum_{p\in \text{Crit}(f;q)}(\triangle_t^{(q)}(\varphi_{U_p}u),\varphi_{U_p}u)\]
Take $C>0$ so that it satisfies Lemmas 4.6, 4.7 and 4.8 such that 
\begin{align*}
  \sum_{U\in S} (\triangle_t (\varphi_{U_p}u),\varphi_{U_p}u)&\geq Ct\|\varphi_V u\|^2+Ct\sum_{p\in\text{Crit}(f)\backslash\text{Crit}(f;q)}\|\varphi_{U_p}u\|^2+C\sum_{p\in \text{Crit}(f;q)}(t\|\varphi_{U_p}u\|^2-\varepsilon^{-\epsilon_0 t}\|u\|^2) \\
  &\geq \frac{t}{\tilde{C}}\|u\|^2-\tilde{C}(1+e^{-\epsilon_0t})\|u\|^2
\end{align*}
Thus,
\[(\triangle_t^{(q)} u,u) \geq \tilde{\tilde{C}}t\|u\|^2.\]
when $t\gg 1$ for some $\tilde{\tilde{C}}>0$.
\end{proof}

We finally arrive at the stage of showing $\dim E^q_{[0,e^{-Ct}],t}(M)\leq m_q$. Assume $\dim E^q_{[0,e^{-Ct}],t}(M)> m_q$, then there exists $u\in E^q_{[0,e^{-Ct}],t}(M)$ which is perpendicular to $\text{span}\beta^\perp$. But this is impossible, by theorem \ref{thm:KeyEstimation}, 
\[(\triangle_t^{(q)}u,u)\geq \tilde{\tilde{C}}t\|u\|^2\] 
when $t\gg 1$. On the other hand, because $u\in E^q_{[0,e^{-Ct}],t}$,
\[(\triangle_t^{(q)}u,u)\leq e^{-Ct}\|u\|^2.\]
Thus, $\tilde{\tilde{C}}t<e^{-Ct}$ when $t\gg 1$, a contradiction. Therefore, 
\[\dim E^q_{[0,e^{-Ct}],t}(M)\leq m_q.\]
\subsection{Eigenvalues of Harmonic Oscillator}

Consider the operator $H=-\frac{\partial}{\partial x}+x^2$ on $L^2(\mathbb{R})$, called the \textbf{Harmonic Oscillator} (In Quantum Mechanics, this operator is the Hamiltonian of a system called the one-dimensional Harmonic Oscillator (although we need to time $\frac{1}{2}$ in natural units), and this is the reason why we adopt this name).

The calculation in the continuing paragraphs will derive the eigenvalues of $H$. The result is 
\[EV(H)=\{2n+1:n\in\mathbb{Z}_{\geq0}\}\]
where $EV(H)$ denotes the set of eigenvalues of $H$.

Let $\{A_n(x)\}$ be the polynomials determined by the formula
\[ \sum_{n=0}^\infty A_n(x)\frac{\alpha^n}{n!}=e^{-\alpha^2+2\alpha x}.\]
Define 
\[\phi_n(x)=(2^nn!)^{-1/2}\pi^{-1/4}A_n(x)e^{-x^2/2}\]
for all $n\in \mathbb{Z}_{\geq 0}$. We claim that $\{\phi_n\}$ contains all the eigenfunctions of the operator $H$. 

To do this, we need to show that it is a complete orthonormal system for $L^2(\mathbb{R})$. We do this by the following steps:
\begin{enumerate}
  \item\textbf{Explicit formulas of $\phi_n$}: we first show the equality
  \[\phi_n(x)=(-1)^n(2^nn!)^{-1/2}\pi^{-1/4}e^{x^2/2}\frac{d^n}{dx^n}e^{-x^2}\]
  holds. By generalized Cauchy integral formula, 
  \begin{align*}
    \frac{d^n}{d\alpha^n}\bigg\rvert_{\alpha=0} &e^{-\alpha^2+2\alpha x}=\frac{n!}{2\pi i}\int_{\gamma} \frac{\exp(-\alpha^2+2\alpha x)}{\alpha^{n+1}} d\alpha 
   \end{align*}
   replacing $\alpha$ with $\zeta-x$ then gives
   \begin{align*}
    &=e^{x^2}(-1)^n\frac{n!}{2\pi i}\int_{\gamma'} \frac{-\zeta^2}{(\zeta-x)^{n+1}}d\zeta =(-1)^n e^{x^2}\frac{d^n}{dx^n}e^{-x^2},
  \end{align*}
  where $\gamma$ and $\gamma'$ are anticlockwise loop around $0$ and $x$, respectively. Therefore, the proof follows from the Taylor expansion of $e^{-\alpha^2+2\alpha x}$.
  \item \textbf{Recursion formulas of $A_n$}: note that the  polynomials $\{A_n\}$ satisfies the recursion formulas 
  \begin{align} \label{E:recursion}
  \begin{split}
    A_0&=1 \\
    A_n'&=2nA_{n-1} \\ 
    A_{n+1}&=2x A_{n}-2nA_{n-1}.
  \end{split}
  \end{align}
  for $n\geq 1$. Formulas (\ref{E:recursion}) can be obtained directly: from 1, we know that
  \[A_n(x)=(-1)^ne^{x^2}\frac{d^n}{dx^n}e^{-x^2}.\]
  Applying derivation on both sides gives
  \begin{align}\label{E:A'_n}
  \begin{split}
    A'_n(x)&=(-1)^n\frac{d}{dx}\left(e^{x^2}\frac{d^n}{dx^n}e^{-x^2}\right) \\
    &=(-1)^n2x e^{x^2}\frac{d^n}{dx^n}e^{-x^2}+(-1)^n e^{x^2}\sum_{k=0}^n\binom{n}{k} \frac{d^k}{dx^k}(-2x)\frac{d^{n-k}}{dx^{n-k}}e^{-x^2} \\
    &=(-1)^n2x e^{x^2}\frac{d^n}{dx^n}e^{-x^2}-(-1)^n2x e^{x^2}\frac{d^n}{dx^n}e^{-x^2}+(-1)^{n-1}2n e^{x^2}\frac{d^{n-1}}{dx^{n-1}} e^{-x^2} \\
    &=2nA_{n-1}(x).
  \end{split}
  \end{align}
On the other hand,
  \begin{align*}
  A'_n(x)&=(-1)^n\frac{d}{dx}(e^{x^2}\frac{d^n}{dx^n}e^{-x^2}) \\
  &=(-1)^n2x e^{x^2} \frac{d^n}{dx^2}e^{-x^2}+(-1)^ne^{x^2}\frac{d^{n+1}}{dx^{n+1}}e^{-x^2} \\
  &=2xA_n(x)-A_{n+1}(x).
  \end{align*}
  By formula (\ref{E:A'_n}), we can replace $A_n'(x)$ by $2nA_{n-1}(x)$, and this completes our derivation of (\ref{E:recursion}).
  
  Now, by the recursion formulas (\ref{E:recursion}) and 1,
  \begin{align*}
    &H\phi_n+(x^2-(2n+1))\phi_n \\
    =&(2^nn!)^{-1/2}\pi^{-1/4}(-e^{-x^2/2}(x^2-1)A_n +2xe^{-x^2/2}A'_n-e^{-x^2/2}A''_n+(x^2-(2n+1))e^{-x^2/2}) \\
    =&e^{-x^2/2}(2^nn!)^{-1/2}\pi^{-1/4}(4xnA_{n-1}-4n^2A_{n-2}-2nA_n)=\frac{e^{-x^2/2}}{2^nn!}(2nA_n-2nA_n)=0.
  \end{align*}
  This gives $2n+1$ an eigenvalue of $H$ with $\phi_n$ the corresponding eigenfunction.
  \item\textbf{Orthogonality}: for any two non-negative integers $n$ and $m$, define 
  \[I_{n,m}=\int_\mathbb{R} e^{-x^2}A_n(x)A_m(x).\]
  We claim that $I_{n,m}=2^nn!\sqrt{\pi}\delta_{n,m}$: from recursion formulas (\ref{E:recursion}), one can derive immediately that for $n\geq 1$,
  \begin{align}\label{E:hermite-differential-equation}
  A_n''(x)-2xA_n'(x)+2nA_n(x)=0
  \end{align}
  the verification is left to the readers. Formula (\ref{E:hermite-differential-equation}) can be also written as
  \[e^{x^2}\frac{d}{dx}(e^{-x^2}A'_n(x))+2nA_n=0.\]

  Assume $n\neq m$. We can then deduce the following to equations hold:
  \begin{align*}
  e^{x^2}A_m\frac{d}{dx}(e^{-x^2}A'_n(x))+2nA_nA_m=0 \\
  e^{x^2}A_n\frac{d}{dx}(e^{-x^2}A'_m(x))+2mA_nA_m=0
  \end{align*}
  Subtracting them then gives
  \begin{align}\label{E:A_nA_m}
  2(n-m)e^{-x^2}A_nA_m=\left(A_n\frac{d}{dx}(e^{-x^2}A'_m(x))-A_m\frac{d}{dx}(e^{-x^2}A'_n(x))\right).
  \end{align}
  By product rule of differentiation, 
  \begin{align*}
  A_n\frac{d}{dx}(e^{-x^2}A'_m(x))=\frac{d}{dx}(A_ne^{-x^2}A'_m(x))-A'_ne^{-x^2}A'_m(x) \\
  A_m\frac{d}{dx}(e^{-x^2}A'_n(x))=\frac{d}{dx}(A_me^{-x^2}A'_n(x))-A_m'e^{-x^2}A'_n(x)
  \end{align*} 
  Substituting them to formula (\ref{E:A_nA_m}) and integrating on $\mathbb{R}$ give us that
  \begin{align*}
  2(n-m)\int_{\mathbb{R}}e^{-x^2}A_n(x)A_m(x) dx
  &=\int_{\mathbb{R}}\frac{d}{dx}(A_ne^{-x^2}A'_m(x)-A_me^{-x^2}A'_n(x))dx \\
  &=A_ne^{-x^2}A'_m(x)-A_me^{-x^2}A'_n(x)\bigg|_{-\infty}^\infty=0
  \end{align*}
  and thus $I_{n,m}=0$ when $n\neq m$. 
  
  Now, it remains to show $I_{n,n}=2^nn!\sqrt{\pi}$ for all $n\geq 0$. Note that by recursion formulas (\ref{E:recursion})
  \begin{align*}
  0=I_{n-1,n+1}&=\int_\mathbb{R} e^{-x^2} A_{n-1}(x)A_{n+1}(x)dx \\
  &=\int_\mathbb{R} 2xe^{-x^2} A_{n-1}(x)A_{n}(x)dx-2nI_{n-1,n-1}
  \end{align*}
  or
  \begin{align}\label{E:4.28}
  2nI_{n-1,n-1}=\int_\mathbb{R} 2xe^{-x^2} A_{n-1}(x)A_{n}(x)dx
  \end{align}
  Because $A_n(x)=(-1)^ne^{x^2}\frac{d^n}{dx^n}e^{-x^2}$, the right hand side of (\ref{E:4.28}) equals to
  \begin{align}\label{E:4.29}
  -\int_\mathbb{R} 2x e^{x^2} \frac{d^n}{dx^n}e^{-x^2}\frac{d^{n-1}}{dx^{n-1}}e^{-x^2} dx
  \end{align}
  By product rule of differentiation,
  \[2xe^{x^2}\frac{d^{n-1}}{dx^{n-1}}e^{-x^2}=\frac{d}{dx}(e^{x^2}\frac{d^{n-1}}{dx^{n-1}}e^{-x^2})-e^{x^2}\frac{d^n}{dx^n}e^{-x^2}.\]
  Pluging in (\ref{E:4.29}) and observing that $A_n(x)=(-1)^ne^{x^2}\frac{d^n}{dx^n}e^{-x^2}$, we then derive
  \begin{align*}
  2nI_{n-1,n-1}&=\int_\mathbb{R} e^{x^2}\frac{d^n}{dx^n}e^{-x^2}\frac{d^n}{dx^n}e^{-x^2}dx-\int_\mathbb{R}\frac{d^n}{dx^n}e^{-x^2}\frac{d}{dx}(e^{x^2}\frac{d^{n-1}}{dx^{n-1}}e^{-x^2})dx \\
  &=\int_\mathbb{R}e^{-x^2}A_n(x)A_n(x)dx+e^{-x^2}A_n(x)A_{n-1}(x)\bigg|_{-\infty}^\infty+\int_\mathbb{R}e^{-x^2}A_{n-1}(x)A_{n+1}(x)dx \\
  &=I_{n,n}
  \end{align*}
  This recursive result immediately gives us
  \[I_{n,n}=2^nn!I_{0,0}=2^nn!\sqrt{\pi}\]
  due to the fact that
  \[I_{0,0}=\int_\mathbb{R}e^{-x^2}dx=\sqrt{\pi}.\]
  To sum up, we have proven that $I_{n,m}=2^nn!\sqrt{\pi}\delta_{n,m}$. As
  \[\int_\mathbb{R} \phi_n(x)\phi_m(x)=(2^nn!)^{-1/2}\pi^{-1/4}(2^mm!)^{-1/2}\pi^{-1/4}I_{n,m}=\delta_{n,m},\]
  we conclude that the system $\{\phi_n\}$ is orthogonal.
  \item\textbf{Completeness}: to show completeness of the system $\{\phi_n\}$, it suffices to show $f\in L^2(\mathbb{R})$ is orthogonal to the system $\{\phi_n\}$ if and only if $f=0$ almost everywhere. 
  
  For any $f\in L^2(\mathbb{R})$ such that $(f,\phi_n)=0$ for all $n\in\mathbb{Z}_{\geq 0}$, we claim that $f=0$ almost everywhere. Note that each 
  \[
  P_n(x)=e^{x^2}\frac{d^ne^{u^2}}{d u^n}(x)
  \]
is a degree $n$ polynomial. Therefore, the set ${P_n}$ is a basis for the linear space of polynomial. This means
  \[\int_\mathbb{R} f(x)\phi_n(x)=(-1)^n(2^n n!)^{-1/2}\pi^{-1/4}\int_\mathbb{R} f(x)e^{-x^2/2}P_n(x) dx=0\]
  if and only if
  \[\int_\mathbb{R} f(x)e^{-x^2/2} x^n dx=0.\]
  As a result, the entire function
  \begin{align}\label{E:FTofF}
  F(z)=\int_{\mathbb{R}} f(x)e^{-\frac{x^2}{2}+xz}dx=\sum_{n=0}^\infty \int_\mathbb{R} f(x) e^{-x^2/2} x^n dx=0.
  \end{align}
  If we take $z=-it$, then (\ref{E:FTofF}) says $\widehat{fe^{-x^2/2}}=0$. Fourier inversion formula then gives $f=0$ almost everywhere, which means the system $\{\phi_n\}$ is complete.
\end{enumerate}
Denote $L^2_q(\mathbb{R}^n)$ as the set of square integrable $q$-form on $\mathbb{R}^n$. Define an operator $\Box_t^{(r)}$ acting on $L^2_q(\mathbb{R}^n)$ by
\begin{align}\label{E:Boxt}
  \Box_t udx_J=(-\triangle u+t^2|x|^2u)dx_J-tu\sum_{j=1}^r [dx_j\wedge,dx_j\lrcorner](dx_J)+tu\sum_{j=r+1}^n [dx_j\wedge,dx_j\lrcorner](dx_J)
\end{align}
where $\triangle$ denotes the standard Laplacian on $\mathbb{R}^n$ and $|J|=q$.

  Since $[dx_j\wedge,dx_j\lrcorner](dx_J)=\varepsilon_J^j$, where
  \begin{align*}
    \varepsilon_J^j=
    \begin{cases}
    1 & \text{if $j\in J$}\\
    -1 & \text{if $j \not\in J$}
    \end{cases}
  \end{align*}
  the eigenvalues of $\Box_t$ on $L^2_q(\mathbb{R}^n)$ are  
  \begin{align}\label{E:Boxt-eigenvalues}
  \left\{t\sum_{j=1}^r (2N_j+1-\varepsilon_J^j)+t\sum_{j=r+1}^n (2N_j+1+\varepsilon_J^j):(N_1,\ldots,N_n)\in \mathbb{Z}_{\geq 0}\text{ and } |J|=q\right\}
  \end{align}
  So we get
  \begin{align}
    \ker(\Box_t^{(r)})=
    \begin{cases}
      0 & \text{if $r\neq q$} \\
      \mathbb{R} e^{-tx^2}dx_1\wedge\cdots\wedge dx_q & \text{if $r=q$}
    \end{cases}
  \end{align}
  The eigenvalues other than $0$ are all of the form $Kt$ for some $K>0$.

\end{document}